\documentclass{article}%
\usepackage{graphicx}
\usepackage{amsmath}
\usepackage{amsfonts}
\usepackage{amssymb}%
\providecommand{\U}[1]{\protect\rule{.1in}{.1in}}
\newtheorem{theorem}{Theorem}
{}

{}

\newtheorem{summary}{Summary}
\newenvironment{proof}[1][Proof]{\textbf{#1.} }{\ \rule{0.5em}{0.5em}}

\begin{document}

\title{On the Bloch eigenvalues and spectrum of the differential operators of odd order}
\author{O. A. Veliev\\{\small Dogus University, Esenkent 34775,\ Istanbul, Turkey.}\\\ {\small e-mail: oveliev@dogus.edu.tr}}
\date{}
\maketitle

\begin{abstract}
In this paper we consider the Bloch eigenvalues and spectrum of the
non-self-adjoint differential operator $L$ generated \ by the differential
expression of odd order $n$ with the periodic PT-symmetric coefficients, where
$n>1.$ We study the localizations of the Bloch eigenvalues and the structure
of the spectrum. Moreover, we find conditions on the norm of the coefficients
under which the spectrum of $L$ coincides with the real line.



\end{abstract}

\section{Introduction}

Let $L$ be the differential operator generated in the space $L_{2}%
(-\infty,\infty)$ by the differential expression
\begin{equation}
l(y)=(-i)^{n}y^{(n)}(x)+%
%TCIMACRO{\tsum \limits_{v=2}^{n}}%
%BeginExpansion
{\textstyle\sum\limits_{v=2}^{n}}
%EndExpansion
(-i)^{n-v}p_{v}(x)y^{(n-v)}(x), \tag{1}%
\end{equation}
where $n$ is an odd integer greater than $1$ and $p_{v}$ for $v=2,3,...n$ are
$1$-periodic PT-symmetric function satisfying $\left(  p_{v}\right)
^{(n-v)}\in L_{2}\left[  0,1\right]  $. It is well-known that (see [4, 5]) the
spectrum $\sigma(L)$ of the operator $L$ is the union of the spectra of the
operators $L_{t}$ for $t\in(-1,1]$ generated in $L_{2}\left[  0,1\right]  $ by
(1) and the boundary conditions
\begin{equation}
y^{(\mathbb{\nu})}\left(  1\right)  =e^{i\pi t}y^{(\mathbb{\nu})}\left(
0\right)  \tag{2}%
\end{equation}
for $\mathbb{\nu}=0,1,...,(n-1).$ The spectra $\sigma(L_{t})$ of the operators
$L_{t}$ consist of the eigenvalues called the Bloch eigenvalues of $L$.

The operators $L$ and $L_{t}$ are denoted by $L(0)$ and $L_{t}(0)$ if
$p_{2},p_{3},...,p_{n},$ are the zero functions. It is clear that $\left(
2\pi k+\pi t\right)  ^{n}$ and $e^{i\pi\left(  2k+t\right)  x}$ for
$k\in\mathbb{Z}$ are respectively the eigenvalues and eigenfunctions of
$L_{t}(0).$ The numbers $\left(  2\pi k+\pi t\right)  ^{n}$ for $k\in
\mathbb{Z}$ are the simple eigenvalues of $L_{t}(0)$ and the set of all Bloch
eigenvalues of $L(0)$ cover only $1$ times the real axis.

In [10] we proved that if the coefficients of (1) are the $m\times m$ matrices
with the PT-symmetric entries and $m$ is an odd number, then $\mathbb{R\subset
}\sigma(L).$\ In this paper we consider the case $m=1$ in detail and prove
that the nonreal part $\sigma(L)\backslash\mathbb{R}$ of $\sigma(L)$ is
contained in the rectangle
\begin{equation}
\left\{  \lambda\in\mathbb{C}:\left\vert \operatorname{Re}\lambda\right\vert
\leq(2\pi N)^{n},\text{ }\left\vert \operatorname{Im}\lambda\right\vert
<\frac{\sqrt{10}}{3}\left(  2N+1\right)  ^{n-3/2}\pi^{n-2}C\right\}  , \tag{3}%
\end{equation}
where $N$ is the smallest integer satisfying $N\geq\pi^{-2}C+1$ and%
\[
C=%
%TCIMACRO{\tsum \limits_{v=2}^{n}}%
%BeginExpansion
{\textstyle\sum\limits_{v=2}^{n}}
%EndExpansion%
%TCIMACRO{\tsum \limits_{s=0}^{n-v}}%
%BeginExpansion
{\textstyle\sum\limits_{s=0}^{n-v}}
%EndExpansion
\frac{\left(  n-v\right)  !\left\Vert \left(  p_{v}\right)  ^{(s)}\right\Vert
}{s!(n-v-s)!\pi^{v+s-2}}.
\]
Moreover, we prove that if $C\leq\ \pi^{2}2^{-n+1/2}$, then $\sigma
(L)=\mathbb{R}$.

Note that the obtained results and the methods of the investigations for the
odd case $(n=2v+1)$ essentially differ from the results and the methods of the
investigations for the even case $(n=2v)$. The general even case is similar to
the case $n=2$ (Schr\"{o}dinger operator). There are a large number of papers
for the Schr\"{o}dinger operator (see the monographs [1, Chapters 4 and 6] and
[8, Chapters 3 and 5] and the papers they refer to). The results and the
method used in this paper are completely different from the results and
methods of those papers. That is why, and in order not to deviate from the
purpose of this paper, we do not discuss them here, in detail. We only note
that, in the first paper [2] about the PT-symmetric periodic potential, the
disappearance of real energy bands for some complex-valued PT-symmetric
periodic potentials have been reported. In [6] it was showed that the
disappearance of such real energy bands implies the existence of nonreal band
spectra. In [7], I proved that the main part of the spectrum of $L$ is real
and contains the large part of $[0,\infty).$ However, in general, the spectrum
contains also infinitely many nonreal acts.

\section{Main Results}

First we investigate the localizations of the eigenvalues of $L_{t}$. This
investigation is similar to Section 2 of [9], where the self-adjoint case was
investigated. We prove that if $\left\vert k\right\vert \geq\pi^{-2}C+1$, then
the $\delta_{k}(t):=\frac{3}{2}\pi^{n-2}C\left\vert (2k+t)\right\vert ^{n-2}$
neighborhood
\begin{equation}
U(k,t)=\left\{  \lambda\in\mathbb{C}:\left\vert \lambda-(2\pi k+\pi
t)^{n}\right\vert <\delta_{k}(t)\right\}  \tag{4}%
\end{equation}
of the eigenvalue $(2\pi k+\pi t)^{n}$ of $L_{t}(0)$ contains only one
eigenvalue of $L_{t},$ where $C$ is defined in (3). To prove this statement we
use the following formulas. Let $\lambda(k,t,\varepsilon)$ be the eigenvalue
of the operator $L_{t,\varepsilon}:=L_{t}(0)+\varepsilon(L_{t}-L_{t}%
(0))~$satisfying $\operatorname{Re}\left(  \lambda(k,t,\varepsilon)\right)
\in I(k,t)~\ $and $\Psi_{\lambda(k,t,\varepsilon)}$ be a normalized
eigenfunction of $L_{t,}$ corresponding to the eigenvalue $\lambda
(k,t,\varepsilon),$ where $\varepsilon\in\lbrack0,1]$ and
\begin{equation}
I(k,t)=[(2\pi k+\pi t-\pi)^{n},(2\pi k+\pi t+\pi)^{n}). \tag{5}%
\end{equation}
Sometimes, for brevity, instead of $\Psi_{\lambda(k,t,\varepsilon)}$ and
$\lambda(k,t,\varepsilon)$ we write $\Psi_{\lambda}$ and $\lambda.$

Multiplying the equation $L_{t,\varepsilon}\Psi_{\lambda}=\lambda\Psi
_{\lambda}$ by $e^{i\left(  2\pi k+\pi t\right)  x}$ and using the equality
\[
L_{t}(0)e^{i\left(  2\pi k+\pi t\right)  x}=\left(  2\pi k+\pi t\right)
^{n}e^{i\left(  2\pi k+\pi t\right)  x}%
\]
we get
\begin{equation}
\left(  \lambda-\left(  2\pi k+\pi t\right)  ^{n}\right)  \left(
\Psi_{\lambda},e^{i\left(  2\pi k+\pi t\right)  x}\right)  =\varepsilon
\sum\limits_{\nu=2}^{n}(p_{v}\Psi_{\lambda}^{(n-v)},e^{i\left(  2\pi k+\pi
t\right)  x}), \tag{6}%
\end{equation}
where $(\cdot,\cdot)$ is the inner product in $L_{2}\left[  0,1\right]  .$ To
prove that $\lambda(k,t,\varepsilon)\in U(k,t)$ we estimate the right side of
(6) and $\left(  \Psi_{\lambda},e^{i\left(  2\pi k+\pi t\right)  x}\right)  $.
First let us estimate the right hand side of (6). For this we use the
integration by parts formula and obtain
\begin{equation}
(p_{v}\Psi_{\lambda(k,t,\varepsilon)}^{(n-v)},e^{i\left(  2\pi k+\pi t\right)
x})=(\Psi_{\lambda(k,t,\varepsilon)},\left(  \overline{p_{v}}e^{i\left(  2\pi
k+\pi t\right)  x}\right)  ^{(n-v)}). \tag{7}%
\end{equation}
If $k\in\mathbb{Z}\backslash\left\{  0\right\}  ,$ then by direct calculations
one can easily verify that%
\begin{equation}
\left\Vert \left(  \overline{p_{v}}e^{i\left(  2\pi k+\pi t\right)  x}\right)
^{(n-v)}\right\Vert \leq%
%TCIMACRO{\tsum \limits_{s=0}^{n-v}}%
%BeginExpansion
{\textstyle\sum\limits_{s=0}^{n-v}}
%EndExpansion
\frac{\left(  n-v\right)  !\left\vert 2\pi k+\pi t\right\vert ^{n-v-s}%
\left\Vert p_{v}^{(s)}\right\Vert }{s!(n-v-s)!}. \tag{8}%
\end{equation}
Using (7), (8), Schwarz's inequality and equality $\left\Vert \Psi
_{\lambda(k,t,\varepsilon)}\right\Vert =1$ we get%
\[
\left\vert (p_{v}\Psi_{\lambda(k,t,\varepsilon)}^{(n-v)},e^{i\left(  2\pi
k+\pi t\right)  x})\right\vert \leq\sum\limits_{s=0}^{n-v}\frac{\left(
n-v\right)  !\left\vert 2\pi k+\pi t\right\vert ^{n-v-s}\left\Vert p_{v}%
^{(s)}\right\Vert }{s!(n-v-s)!}\leq
\]%
\[
\left\vert 2\pi k+\pi t\right\vert ^{n-2}\sum\limits_{s=0}^{n-v}\frac{\left(
n-v\right)  !\left\Vert \left(  p_{v}\right)  ^{(s)}\right\Vert }%
{s!(n-v-s)!\left\vert 2\pi k+\pi t\right\vert ^{v+s-2}}.
\]
Therefore, if $k\neq0,$ then
\begin{equation}
\left\vert \varepsilon\sum\limits_{\nu=2}^{n}(p_{v}\Psi_{\lambda
(k,t,\varepsilon)}^{(n-v)},e^{i\left(  2\pi k+\pi t\right)  x})\right\vert
\leq\left\vert 2\pi k+\pi t\right\vert ^{n-2}C \tag{9}%
\end{equation}
for all $\varepsilon\in\lbrack0,1]$ and $t\in(-1,1],$ where $C$ is defined in
(3). In case $k=0$ we have
\begin{equation}
\left\vert \varepsilon\sum\limits_{\nu=2}^{n}(p_{v}\Psi_{\lambda
(k,t,\varepsilon)}^{(n-v)},e^{i\left(  \pi t\right)  x})\right\vert \leq
\pi^{n-2}C \tag{10}%
\end{equation}
for all $\varepsilon\in\lbrack0,1]$ and $t\in(-1,1].$

Let us now show that the estimate for $\left(  \Psi_{\lambda},e^{i\left(  2\pi
k+\pi t\right)  x}\right)  $ can be performed by repeating the calculations
performed in [9] for estimate (11) of $\left(  \Psi_{\lambda},\varphi
_{k,t}\right)  .$ We consider the case $k>0.$ The case $k<0$ can be considered
in the same way. It follows from the definition of $\lambda(k,t,\varepsilon)$
and (5) that
\[
\left\vert \lambda(k,t,\varepsilon)-(2\pi p+\pi t)^{n}\right\vert >\left\vert
(2\pi k+\pi t+\pi)^{n}-(2\pi p+\pi t)^{n}\right\vert
\]
for $p>k$ and
\[
\left\vert \lambda(k,t,\varepsilon)-(2\pi p+\pi t)^{n}\right\vert
\geq\left\vert (2\pi k+\pi t-\pi)^{n}-(2\pi p+\pi t)^{n}\right\vert
\]
for $p<k$. Using these inequalities and the relations obtained from (6) and
(9) by replacing $k$ with $p\neq0$ we get
\begin{equation}
\left\vert \left(  \Psi_{\lambda(k,t,\varepsilon)},e^{i\left(  2\pi p+\pi
t\right)  x}\right)  \right\vert ^{2}<\frac{\pi^{-4}C^{2}\left(  \left\vert
2p+t\right\vert ^{n-2}\right)  ^{2}}{\left(  (2k+t+1)^{n}-(2p+t)^{n}\right)
^{2}} \tag{11}%
\end{equation}
for $p>k$ and
\begin{equation}
\left\vert \left(  \Psi_{\lambda(k,t,\varepsilon)},e^{i\left(  2\pi p+\pi
t\right)  x}\right)  \right\vert ^{2}\leq\frac{\pi^{-4}C^{2}\left(  \left\vert
2p+t\right\vert ^{n-2}\right)  ^{2}}{\left(  (2k+t-1)^{n}-(2p+t)^{n}\right)
^{2}} \tag{12}%
\end{equation}
for $p<k.$ In case $p=0,$ instead of (9) using (10) we get the formulas
\begin{equation}
\left\vert \left(  \Psi_{\lambda(k,t,\varepsilon)},e^{i\pi tx}\right)
\right\vert ^{2}<\frac{\pi^{-4}C^{2}}{\left(  (2k+t+1)^{n}-t^{n}\right)  ^{2}}
\tag{13}%
\end{equation}
for $k<0$ and
\begin{equation}
\left\vert \left(  \Psi_{\lambda(k,t,\varepsilon)},e^{i\pi tx}\right)
\right\vert ^{2}\leq\frac{\pi^{-4}C^{2}}{\left(  (2k+t-1)^{n}-t^{n}\right)
^{2}} \tag{14}%
\end{equation}
for $k>0$ instead of (11) and (12). Note that formulas (11)-(14) coincides
with the formulas (18)-(21) of [9] if $C$ and $e^{i\left(  2\pi p+\pi
t\right)  x}$ are replaced by $M$ and $\varphi_{p,t}$. Therefore, instead of
(18)-(21) of [9] using (11)-(14) and repeating the proof of (11) of [9] we
get
\begin{equation}
\left\vert \left(  \Psi_{\lambda(k,t,\varepsilon)},e^{i\left(  2\pi k+\pi
t\right)  x}\right)  \right\vert >\frac{2}{3}. \tag{15}%
\end{equation}
Thus, instead of (11), (8) and (9), (10) of [9] using respectively (15), (6)
and (9), (10) and repeating the proofs of Theorem 1$(a)$ and $(b)$ of [9] we obtain.

\begin{theorem}
Let $N$ be the smallest integer satisfying $N\geq\pi^{-2}C+1$ and
\[
S(N,t)=\left\{  \lambda\in\mathbb{C}:\operatorname{Re}\lambda\in\lbrack(-2\pi
N+\pi+\pi t)^{n},(2\pi N-\pi+\pi t)^{n})\right\}
\]

$(a)$ If $\left\vert k\right\vert \geq N,$ then the eigenvalues of
$L_{t,\varepsilon}$ for $\varepsilon\in\lbrack0,1]$ lying in the strip
\[
P(k,t)=\left\{  \lambda\in\mathbb{C}:\operatorname{Re}\lambda\in
I(k,t)\right\}
\]
is contained in $U(k,t),$ where $U(k,t)$ and $I(k,t)$ are defined in (4) and (5).

$(b)$ The closures of $S(N,t)$ and $U(k,t)$ for $\left\vert k\right\vert \geq
N$ are pairwise disjoint closed sets.
\end{theorem}

Similarly, repeating the proof of Theorem 2 of [9] we get

\begin{theorem}
If $C\leq\ \pi^{2}2^{-n+1/2}$, then the eigenvalues of $L_{t,\varepsilon}$ for
$\varepsilon\in\lbrack0,1]$ are contained in the disks
\[
U(0,t)=\left\{  \lambda\in\mathbb{C}:\left\vert \lambda-(\pi t)^{n}\right\vert
<\frac{1}{5}\pi^{n}\right\}  ,
\]%
\[
U(1,t)=\left\{  \lambda\in\mathbb{C}:\left\vert \lambda-(2\pi+\pi
t)^{n}\right\vert <\frac{3}{10}\left\vert 2+t\right\vert ^{n-2}\pi
^{n}\right\}  ,
\]%
\[
U(-1,t)=\left\{  \lambda\in\mathbb{C}:\left\vert \lambda-(\pi t-2\pi
)^{n}\right\vert <\frac{3}{10}\left\vert t-2\right\vert ^{n-2}\pi^{n}\right\}
,
\]
and $U(k,t)$ for $\left\vert k\right\vert >1$ which are defined by (4). The
closures of these disks are pairwise disjoint closed sets.
\end{theorem}

Now we consider the eigenvalues of $L_{t,\varepsilon}$ lying in the strip
$S(N,t).$

\begin{theorem}
The eigenvalue $\lambda$ of $L_{t,\varepsilon}$ lying in the strip $S(N,t)$ is
contained in the rectangle
\[
R(N,t)=\left\{  \operatorname{Re}\lambda\in A(N,t),\left\vert
\operatorname{Im}\lambda\right\vert <\frac{\sqrt{10}}{3}\left(  2N+1\right)
^{n-3/2}\pi^{n-2}C\right\}  ,
\]
where $A(N,t)=[(-2\pi N+\pi+\pi t)^{n},(2\pi N-\pi+\pi t)^{n}),$ and
$\varepsilon\in\lbrack0,1].$
\end{theorem}

\begin{proof}
Let $\lambda$ be the eigenvalue of $L_{t,\varepsilon}$ lying in $S(N,t)$ and
$\Psi_{\lambda}$ be a normalized eigenfunction corresponding to $\lambda.$
There exists $k\in\lbrack-N,N]$ such that
\begin{equation}
\left\vert \left(  \Psi_{\lambda},e^{i\left(  2\pi k+\pi t\right)  x}\right)
\right\vert =\max_{p\in\lbrack-N,N]}\left\vert \left(  \Psi_{\lambda
},e^{i\left(  2\pi p+\pi t\right)  x}\right)  \right\vert .\tag{16}%
\end{equation}
First, we prove that
\begin{equation}%
%TCIMACRO{\tsum \limits_{p:\left\vert p\right\vert >N}}%
%BeginExpansion
{\textstyle\sum\limits_{p:\left\vert p\right\vert >N}}
%EndExpansion
\left\vert \left(  \Psi_{\lambda},e^{i\left(  2\pi p+\pi t\right)  x}\right)
\right\vert ^{2}<\frac{1}{10}.\tag{17}%
\end{equation}
If $\lambda\in S(N,t),$ then
\[
\left\vert \lambda-(2\pi p+\pi t)^{n}\right\vert >\left\vert (2\pi N+\pi
t-\pi)^{n}-(2\pi p+\pi t)^{n}\right\vert
\]
for $p>N$ and
\[
\left\vert \lambda-(2\pi p+\pi t)^{n}\right\vert \geq\left\vert (-2\pi N+\pi
t+\pi)^{n}-(2\pi p+\pi t)^{n}\right\vert
\]
for $p<-N$. Using these inequalities and repeating the proof of (13) and (14)
of [9] (use the last two inequalities instead of first two inequalities in the
proof of \ Lemma 1 of [9] and repeat the proof of the lemma) we obtain$_{{}}$
\[%
%TCIMACRO{\tsum \limits_{p:p>N}}%
%BeginExpansion
{\textstyle\sum\limits_{p:p>N}}
%EndExpansion
\left\vert \left(  \Psi_{\lambda},e^{i\left(  2\pi p+\pi t\right)  x}\right)
\right\vert ^{2}\leq\frac{5}{64}%
\]
and $_{{}}$
\[%
%TCIMACRO{\tsum \limits_{p:p<-N}}%
%BeginExpansion
{\textstyle\sum\limits_{p:p<-N}}
%EndExpansion
\left\vert \left(  \Psi_{\lambda},e^{i\left(  2\pi p+\pi t\right)  x}\right)
\right\vert ^{2}<\frac{1}{48}.
\]
The last two inequalities give (17). Now, using Parseval's equality, (17) and
(16) we obtain%
\[
\left\vert \left(  \Psi_{\lambda},e^{i\left(  2\pi k+\pi t\right)  x}\right)
\right\vert >\frac{3}{\sqrt{10\left(  2N+1\right)  }}.
\]
On the other hand, if $k\in\lbrack-N,N]$, then by (9) and (10) the right side
of (6) is not greater than $\left(  \left(  2N+1\right)  \pi\right)  ^{n-2}C.$
Therefore from (6) we obtain
\[
\left\vert \lambda-(2\pi k+\pi t)^{n}\right\vert <\frac{\sqrt{10}}{3}\left(
2N+1\right)  ^{n-3/2}\pi^{n-2}C
\]
that gives the proof of the theorem.
\end{proof}

Now using Theorem 1-3 we consider the spectrum of $L.$ Besides, we use the
following results from [10] formulated here as summaries.

\begin{summary}
$(a)$ The real line $\mathbb{R}$ is a subset of $\sigma(L)$ ( Theorem 2$(b)$
of [10])

$(b)$ If $\lambda$ is an eigenvalue of $L_{t},$ then $\overline{\lambda}$ is
also an eigenvalue of $L_{t}$ (This result follows from Theorem 1$(a)$ of [10]).
\end{summary}

Note that Summary 1$(b)$ is a characteristic property of the differential
operators with PT-symmetric coefficients. For the operator $L_{t}$ this result
immediately follows from Theorem 1$(a)$ of [10] due to the following. Theorem
1$(a)$ of [10] states that if $\Psi$ is a solution of $l(y)=\lambda y$, then
the function $\Phi$ defined by $\Phi(x,\lambda)=\overline{\Psi(-x,\lambda)}$
is a solution of $l(y)=\overline{\lambda}y,$ where $l(y)$ is defined in (1).
On the other hand, $\Psi$ satisfies boundary conditions (2) if and only if
$\Psi(x)$ has the form $\Psi(x)=e^{i\pi tx}p(x),$ where $p(x)$ is a periodic
function. Then one can easily verify that the function $\Phi$ has the same
form. Therefore, if $\lambda$ is an eigenvalue of $L_{t},$ then $\overline
{\lambda}$ is also an eigenvalue of $L_{t}.$

Now we are ready to prove the following results of this paper.

\begin{theorem}
$(a)$ Each of the disks $U(k,t)$ for $\left\vert k\right\vert \geq N$ contains
only one eigenvalues of $L_{t}$, where $N$ is defined in Theorem 1. This
eigenvalue is a real number.

$(b)$ The real part $\sigma(L)\cap\mathbb{R}$ of the spectrum $\sigma(L)$ of
$L$ is $\mathbb{R}$ and the nonreal part $\sigma(L)\backslash\mathbb{R}$ of
$\sigma(L)$ consists of the curves lying in the rectangle (3).

$(c)$ If $C\leq\ \pi^{2}2^{-n+1/2}$, then $(a)$ is valid for all
$k\in\mathbb{Z}$ and $\sigma(L)=\mathbb{R}$.
\end{theorem}

\begin{proof}
$(a)$ Since the strips $S(N,t)$ and $P(k,t)$ for $\left\vert k\right\vert \geq
N,$ defined in Theorem 1, is a cover of $\mathbb{C},$ it follows from Theorems
3 and 1$(a)$ that all eigenvalues of $L_{t,\varepsilon}$ for all
$\varepsilon\in\lbrack0,1]$ are contained in the union of the sets $R(N,t)$
and $U(k,t)$ for $\left\vert k\right\vert \geq N.$ Moreover, by Theorem 1
$(b)$ the closures of the disks $U(k,t)$ for $\left\vert k\right\vert \geq N$
and the rectangle $R(N,t)$ are pairwise disjoint closed sets. Therefore there
exists a closed curve $\Gamma(k,t)$ which enclose only the disk $U(k,t)$ and
lies in the resolvent set of the operators $L_{t,\varepsilon}$ for all
$\varepsilon\in\lbrack0,1].$ Since $L_{t,\varepsilon}$ is a halomorphic family
with respect to $\varepsilon,$ we conclude that the number of eigenvalues
(counting the multiplicity) of the operators $L_{t,0}=L_{t}(0)$ and
$L_{t,1}=L_{t}$ lying inside $U(k,t)$ are the same (see [3, Chap. 7]). Since
the operator $L_{t,0}$ has only $1$ eigenvalues lying inside $U(k,t)$, the
operator $L_{t,1}$ also has only $1$ eigenvalues lying inside $U(k,t).$ Thus
$U(k,t)$ for $k\geq N$ contains only $1$ eigenvalues of $L_{t}.$ In the same
way we prove this statement for $k\leq-N.$ If the eigenvalue $\lambda$ of
$L_{t}$ lying in $U(k,t)$ is a nonreal number, then by Summary 1$(b)$
$\overline{\lambda}$ is also an eigenvalue of $L_{t}$ lying in $U(k,t).$ This
contradicts to the first sentence of $(a)$.

$(b)$ It is well known that (see [4, 5]) the spectrum of $L$ consist of the
curves and the points of these curves are the Bloch eigenvalues. Therefore the
proof of $(b)$ follows from $(a),$ Summary 1$(a)$ and Theorem 3.

$(c)$ Using Theorem 2 and repeating the proof of $(a),$ we obtain that if
$C\leq\ \pi^{2}2^{-n+1/2}$, then $(a)$ is valid for all $k\in\mathbb{Z}$, all
Bloch eigenvalues of $L$ are real numbers and $\sigma(L)\subset\mathbb{R}$.
Thus, the proof of $(c)$ follows from Summary 1$(a).$
\end{proof}

\end{document}